\documentclass[11pt, a4paper]{amsart}
\usepackage{times,epsfig}
\usepackage{amsfonts,amscd,amsmath,amssymb,amsthm}

 \newcommand{\margp}[1]{}

\usepackage{bbm}

\numberwithin{equation}{section}

\renewcommand{\phi}{\varphi}

\DeclareSymbolFont{SY}{U}{psy}{m}{n}
\DeclareMathSymbol{\emptyset}{\mathord}{SY}{'306}

\DeclareMathOperator{\rank}{rank}

\DeclareMathOperator{\Ran}{Ran}

\DeclareMathSymbol{\newtimes}{\mathbin}{SY}{'264}

\newcommand{\TT}{\wt T} 
\newcommand{\HH}{\wt \cH} 
\newcommand{\hatP}{\wt P} 
\newcommand{\wt}{\widetilde}
\newcommand\weaktends {\overset {\text {\rm WOT} }{\longrightarrow}}
\newcommand\strongtends {\overset {\text {\rm SOT} }{\longrightarrow}}
\newcommand{\cHknot}{\cH_0}

\renewcommand{\phi}{\varphi}
\newcommand{\de}{\delta}

\newcommand{\T}{\mathbb{T}}
\newcommand{\C}{\mathbb{C}}
\newcommand{\Z}{\mathbb{Z}}
\newcommand{\N}{\mathbb{N}}

\newcommand{\cA}{{\mathcal A}}
\newcommand{\cB}{{\mathcal B}}

\newcommand{\cF}{{\mathcal F}}
\newcommand{\cH}{{\mathcal H}}

\newcommand{\cK}{{\mathcal K}}
\newcommand{\gothh}{\mathfrak h}

\newcommand{\cKK}{{\mathcal R}}        
\newcommand{\cL}{{\mathcal L}}

\newcommand{\cO}{{\mathcal O}}

\newcommand{\cR}{{\mathcal R}}
\newcommand{\cS}{{\mathcal S}}
\newcommand{\cT}{{\mathcal T}}

\newcommand{\cW}{{\mathcal W}}

\newcommand{\eps}{\varepsilon} 

\newcommand\beqn{\begin{equation}}
\newcommand\neqn{\end{equation}}

\newcommand\cPfin{{\mathcal P}_\text{fin}}

\newcommand{\1}{\mathbbm 1}

{\bf}{\it}
{\bf}{\it}
{\bf}{\it}
{\bf}{\it}
{\bf}{\it}
\newtheorem{theorem}{Theorem}[section]{\bf}{\it}
\newtheorem*{vartheorem}{Theorem}
\newtheorem{proposition}[theorem]{Proposition}{\bf}{\it}
\newtheorem{corollary}[theorem]{Corollary}{\bf}{\it}
{\it}{\rm}
\newtheorem{lemma}[theorem]{Lemma}{\bf}{\it}
\newtheorem{remark}[theorem]{Remark}{\it}{\rm}
{\it}{\rm}
\newtheorem{definition}[theorem]{Definition}{\bf}{\it}
{\bf}{\it}
{\bf}{\it}
{\bf}{\it}

\DeclareMathAlphabet{\Ma}{U}{msa}{m}{n}
\DeclareMathAlphabet{\Mb}{U}{msb}{m}{n}

\DeclareMathAlphabet{\Meuf}{U}{euf}{m}{n}
\DeclareSymbolFont{ASMa}{U}{msa}{m}{n}
\DeclareSymbolFont{ASMb}{U}{msb}{m}{n}
\DeclareMathSymbol{\hrist}{\mathord}{ASMa}{"16}
\DeclareMathSymbol{\varkappa}{\mathalpha}{ASMb}{"7B}
\DeclareMathSymbol{\CrPr}{\mathord}{ASMb}{"6F}

\def\got#1{\Meuf{#1}}
\def\ot #1.{{\got{#1}}}

\title[F\o lner sequences and finite operators]
{\flushleft {  \today} \\
\vskip.5cm
F\o lner sequences and finite operators
}

\author{Fernando Lled\'o}
\address{Department of Mathematics, University Carlos~III Madrid,
  Avda.~de la Universidad~30, E-28911 Legan\'es (Madrid), Spain
  and Instituto de Ciencias Matem\'{a}ticas (CSIC - UAM - UC3M - UCM)}
\email{flledo@math.uc3m.es}

\author{Dmitry V. Yakubovich}
\address{Departamento de Matem\'{a}ticas,
Universidad Autonoma de Madrid, Cantoblanco 28049 (Madrid) Spain
\enspace and
\newline
\phantom{r} Instituto de Ciencias
Matem\'{a}ticas (CSIC - UAM - UC3M - UCM)}
\email{dmitry.yakubovich@uam.es}


\thanks{
The first-named author is partly supported by the project MTM2009-12740-C03-01
of the Spanish Ministry of Science and Innovation.
Part of this work was done while one of use (F. Ll.) was visiting the
Centre de Recerca Matem\`atica in Barcelona.
The second-named author has been supported by the
Project MTM2008-06621-C02-01, DGI-FEDER,
of the Ministry of Science and Innovation of Spain.
Both authors acknowledge the support
by the ICMAT Severo Ochoa project SEV-2011-0087
of the Ministry of Economy and Competition of Spain.
}

\begin{document}

\begin{abstract}
%
%
%
This article analyzes F\o lner sequences of projections for bounded linear operators
and their relationship
to the class of finite operators introduced by Williams in the 70ies.
We prove that each essentially hyponormal operator has a proper F\o lner sequence
(i.e. a F\o lner sequence of projections strongly converging to $\1$).
In particular, any quasinormal, any subnormal, any hyponormal and any
essentially normal operator has a proper F\o lner sequence.
Moreover, we show that an operator is finite if and only if it has a
proper F\o lner sequence
or if it has a non-trivial finite dimensional reducing subspace.
We also analyze the structure of operators which have
no F\o lner sequence and give examples of them. For this analysis we
introduce the notion of strongly non-F\o lner operators,
which are far from finite block reducible operators, in some uniform sense,
and show that this class coincides with the class of non-finite operators.
\end{abstract}
\maketitle

\tableofcontents


\section{Introduction}\label{sec:intro}

The notion of a F\o lner sequence was introduced in the context of groups
to give a new characterization of amenability.
A discrete countable group $\Gamma$ is amenable if it has an invariant
mean, i.e., there is a state $\psi$
on the von Neumann algebra $\ell^\infty(\Gamma)$ such that
\[
  \psi(u_\gamma g)=\psi(g)\;,\quad \gamma\in\Gamma\;,\quad g\in \ell^\infty(\Gamma)\;,
\]
where $u$ is the left-regular representation of $\Gamma$ on $\ell^2(\Gamma)$.
A F{\o}lner sequence for a countable discrete group
$\Gamma$ is an increasing sequence of non-empty finite subsets $\Gamma_n\subset\Gamma$
with $\Gamma=\cup_n \Gamma_n$ and that satisfy
\begin{equation}\label{eq:foelner-g}
\lim_{n}
\frac{|(\gamma \Gamma_n)\triangle
\Gamma_n|}{|\Gamma_n|} =0\qquad\text{for all}\quad \gamma\in\Gamma \,,
\end{equation}
where $\triangle$ denotes the symmetric difference and $|\Gamma_n|$
is the cardinality of $\Gamma_n$. Then, $\Gamma$ has a F{\o}lner sequence
if and only if $\Gamma$ is amenable (cf.~Chapter~4 in \cite{bPaterson88};
see also \cite{Vershik82} for a review stressing the fundamental fact
that amenable groups are those which can be approximated by finite groups).

F\o lner sequences were introduced
in the context of operator algebras by Connes in Section~V of his seminal
paper \cite{Connes76} (see also \cite[Section~2]{ConnesIn76}).
This notion is an algebraic analogue of F\o lner's characterization
of amenable discrete groups and was used by Connes as an essential tool in the
classification of injective type~II$_1$ factors.
If $\mathcal{H}$ is a Hilbert space, we will denote by
$\cL(\cH)$ the algebra of bounded linear operators on $\mathcal{H}$.
In this article all Hilbert spaces will be complex and separable.
\begin{definition}\label{def:Foelner}
Let $\mathcal{H}$ be a
Hilbert space, and let
$\mathcal{T}\subset\cL(\cH)$. A
sequence of non-zero
finite rank orthogonal projections $\{P_n\}_{n\in \N}$ on $\cH$
is called a F{\o}lner sequence for $\mathcal{T}$ if
\begin{equation}\label{eq:F1}
\lim_{n} \frac{\|T P_n-P_n T\|_2}{\|P_n\|_2} = 0\;\;,\quad T\in\cT\; ,
\end{equation}
where $\|\cdot\|_2$ denotes the Hilbert-Schmidt norm.
If the F\o lner sequence $\{P_n\}_{n\in \N}$ satisfies, in addition, that it is increasing
and converges strongly to $\1$, then we say it is a proper F\o lner sequence.
\end{definition}

%

Normal operators, compact operators
and Toeplitz operators with $L^\infty$ symbol are examples of operators with
\margp{LEERLO OTRA VEZ}
a proper F\o lner sequence (cf.~\cite[Chapter~7]{bHagen01}).

To simplify expressions we will often use the following notation:
for $T\in\cL(\cH)$ and $P$ a non-zero finite rank orthogonal projection, put
\begin{equation}
\label{def-phi}
 \phi(T,P):= \frac{\|TP-PT\|_2}{\|P\|_2}\;.
\end{equation}

There is a canonical relation between the group theoretic and operator algebraic
notions of F\o lner sequences in terms of the group algebra.
Let $\Gamma$ be a discrete, countable and amenable group and
$\{\Gamma_n\}_{n\in\N}\subset\Gamma$ a
F\o lner sequence. Denote by $P_n$ the
orthogonal finite-rank projections from
$\ell^2(\Gamma)$ onto $\ell^2(\Gamma_n)$.  Then
$\{P_n\}_{n\in\N}$ is a proper F{\o}lner sequence for the group
C*-algebra of $\Gamma$.
(In Proposition~4 in \cite{Bedos97}, a stronger result is shown:
the sequence $\{P_n\}_{n\in\N}$ mentioned before is a proper F\o lner sequence even for
the group von Neumann algebra, i.e., for the weak operator closure
of the algebra generated by the left-regular representation of
$\Gamma$ on $\ell^2(\Gamma)$. In general, if a
C*-algebra $\cA\subset\cL(\cH)$ has a F\o lner sequence, then
its weak closure needs not have one.)

In addition to these theoretical developments, F\o lner sequences have been
used in spectral approximation problems: given
an operator $T$ on a complex Hilbert space $\cH$
and a sequence of matrices (or linear operators) $\{T_n\}_{n\in\N}$
that approximate $T$ in some sense, a natural question is
whether the spectral objects of $T_n$ tend to those of $T$ as $n$ grows.
There are many references that treat this question from different points
of view. Some standard textbooks that contain many examples and an extensive
list of references are \cite{bAhues01,bChatelin83,bHagen01}.
B\'edos used F\o lner sequences for operators in the context of eigenvalue
distribution problems (cf.~\cite{Bedos97}) and refined
earlier results by Arveson stated in \cite{Arveson94,ArvesonIn94}.
See also the introduction in \cite{pLledo11} and references cited therein.
It is worth mentioning that in the last two decades, the
relation between spectral approximation problems and F\o lner sequences
for non-selfadjoint and non-normal operators has been also
explored, see for instance \cite{Widom94,Tilli99,BottchOtte05,Roch07}.
Notice that in the context of spectral approximation, proper F\o lner sequences
are important. 

The second important concept for this article is the class of
finite operators. They were introduced and analyzed in a
classical article by Williams (cf., \cite{Williams70}). Finite
operators $T\in\cL(\cH)$ on an infinite dimensional Hilbert space
$\cH$ have the property that $0$ is in the closure of the
numerical range of the commutator $TX-XT$ for all $X\in\cL(\cH)$
(see also Section~\ref{sec:foelner} for a formal definition and
additional results). As Williams explains in his article the name
'finite' is admittedly {\em ad hoc}. It comes from the fact that
the class of finite operators contains the closure of all finite
block reducible operators (i.e., operators having a non-trivial
finite dimensional reducing subspace).

The aim of the present paper is to analyze the role of F\o lner sequences in the context of
a single operator and in relation to the class of finite operators.
If $\{P_n\}_{n\in\N}$ is a F\o lner sequence for $T\in\cL(\cH)$,
then 
we have with respect to the decomposition
$\cH=P_n\cH\oplus (P_n\cH)^\perp$
\[
 T\cong\begin{pmatrix}
        T_1^{(n)} &  T_2^{(n)} \\
        T_3^{(n)} &  T_4^{(n)}
       \end{pmatrix}
\quad\mathrm{and}\quad
 TP_n-P_nT\cong
       \begin{pmatrix}
        0 &  -T_2^{(n)} \\
        T_3^{(n)} &  0
       \end{pmatrix}\;.
\]
where $T_1^{(n)}$ is a finite dimensional operator on $P_n\cH$. Therefore
\[
 \frac{\|T P_n-P_n T\|_2}{\|P_n\|_2}=\frac{\|T_2^{(n)}\|_2+\|T_3^{(n)}\|_2}{\|P_n\|_2}\;,
\]
so that the condition in (\ref{eq:F1}) describes the growth of the off-diagonal
blocks of $T$ (in the Hilbert-Schmidt norm) relative to the dimension
of the corresponding subspaces as $n\to\infty$.

In Section~\ref{sec:foelner} we will present some consequences of the existence of F\o lner sequences
for operators and give useful characterizations of it. We also mention standard results for the class
of finite operators.

In Section~\ref{sec:NF} we explore the structure of operators without a F\o lner sequence.
We define a strongly non-F\o lner operator as an operator $T$ such that
there is some positive number $\eps$ with $\phi(T,P)\ge \eps$ for all non-zero finite rank projections $P$ on $\cH$.
In a sense, this condition means that $T$ has to be far from the set of finite block reducible
operators. In Theorem~\ref{thm:str-foln}, we show that any operator with no F\o lner sequence is
the direct sum of an operator on a finite-dimensional space (which can be possibly zero)
and a strongly non-F\o lner operator.

In Section~\ref{sec:relation}, we show the natural relation between the
\margp{RELEERLO}
notion of proper F\o lner sequences and the class
of finite operators.
To describe our results with more detail, let us divide
the set of all operators in $\cL(\cH)$ into four mutually
disjoint classes, according to the following table:
\vskip3mm

\begin{table}[h]
\centering 
  \begin{tabular}{| c || c | c | }
    \hline
    \phantom{phantom}   &  Operators with a proper   &    Operators with no proper \\[0.5ex]
    \phantom{phantom}   &  F\o lner sequence         &    F\o lner sequence          \\
                           \hline\hline
Finite block reducible     &             $\mathcal{W}_{0+}$       &                  $\mathcal{W}_{0-}$    \\ 
\hline
Non finite block reducible &            $\mathcal{W}_{1+}$        &          $\mathcal{S}$     
\\ 
    \hline
  \end{tabular}
\vskip.3cm
 \caption{} 
\end{table}
\noindent As we will prove in Theorem~\ref{teo:williams-foelner},
\begin{itemize}
\item[] {\em an operator is finite if and only if it has a proper F\o lner sequence or it is finite block
reducible.}
\end{itemize}
Therefore the class of finite operators is the disjoint union of the
operators in the classes $\mathcal{W}_{0+}$, $\mathcal{W}_{0-}$ and $\mathcal{W}_{1+}$.
Moreover, an operator is strongly non-F\o lner if and only if it is
not finite, i.e., it belongs to $\mathcal{S}$.
This implies that
the class of strongly non-F\o lner operators is open and dense in
$\cL(\cH)$. We refer to Section~\ref{sec:relation} and to the end of
Section~\ref{sec:examples} for more details.

In Section~\ref{sec:with-FS} we analyze several classes
of non-normal operators and show that each operator
from any of these classes has a proper F\o lner sequence.
In Theorem~\ref{main1}, we show that any essentially hyponormal operator
(i.e., any $T$ such that $T^*T-TT^*$ defines a
nonnegative element of the Calkin algebra) has a proper F\o lner sequence.
This implies that important classes of operators, like, e.g., essentially
normal or hyponormal operators, have also proper F\o lner sequences.

%
In Section~\ref{sec:examples}, we give examples of operators which have no F\o lner sequence. In the
example stated in Proposition~\ref{no-Foelner} we use the fact
that the Cuntz algebra is singly generated as a C*-algebra.
In a subsequent paper \cite{pLledoYakubovich12}
we will discuss asymptotic properties of finite square matrices, related to the property of the existence of a
proper F\o lner sequence for an infinite dimensional linear operator.

\section{Basic properties of F\o lner sequences and finite operators}\label{sec:foelner}

In this section we recall the basic definition and results
concerning F\o lner sequences for operators. We will also
discuss some standard properties of finite operators.
%
%
%
%
In what follows, if $T$ is a linear operator on a Hilbert space $\cH$, we
denote by $\|T\|_p$ its norm in the Schatten-von Neumann
class.
We denote by $\cPfin(\cH)$ the set of all non-zero
finite rank orthogonal projections on $\cH$.


The existence of a proper F\o lner sequence for an operator $T$ is a weaker
property than quasidiagonality.
Recall that an operator $T\in\mathcal{L}(\mathcal{H})$
is said to be {\em quasidiagonal} if there exists a sequence of finite rank orthogonal projections
$\{P_n\}_{n\in \N}$ converging strongly to $\1$ and such that
\begin{equation}\label{QD}
\lim_{n}\|T P_n-P_n T\|=0\;.
\end{equation}
This notion was introduced by Halmos in \cite{Halmos70}
(see also \cite{Voiculescu93} for a review that also relates the concept of quasidiagonality
to other fields like, e.g., C*-algebras).
The existence of a proper F\o lner sequence can be
understood as a kind of quasidiagonality condition relative to the growth of
the dimension of the underlying spaces. It can be shown that
if the sequence of non-zero finite rank orthogonal projections
$\{P_n\}_n$ quasidiagonalizes an operator $T$, then
it is also a proper F\o lner sequence for $T$.
Examples of quasidiagonal operators are compact operators, block-diagonal operators
or normal operators.
Abelian C*-algebras or the set of compact operators $\cK(\cH)$ are examples of
quasidiagonal C*-algebras (cf.~\cite{bBrown08}).

The next result collects some easy consequences of the definition of a
(proper) F\o lner sequence for operators.

\begin{proposition}\label{pro-1}
Let $\cT\subset\cL(\cH)$ be a set of operators and $\{P_n\}_{n\in \N}$
a sequence of finite rank orthogonal projections strongly converging to $\1$.
Then we have
\begin{itemize}
 \item[(i)] $\{P_n\}_{n\in \N}$ is a F\o lner sequence for $\cT$ if and only if it is a F\o lner sequence for
$C^*\big(\,\cT\,,\,\,\1 \big)$, where $C^*(\cdot)$ is the C$^*$-algebra generated by its argument. Moreover,
$\{P_n\}_{n\in \N}$ is a proper F\o lner sequence for $\cT$ if and only if it is a proper F\o lner sequence for
\[
  C^*\big(\,\cT\,,\, \cK(\cH)\,,\,\1 \big)\;.
\]
 \item[(ii)] Let $\dim \cH=\infty$ and $\{P_n\}_{n\in \N}$ be a proper F{\o}lner sequence for $\cT$.
  Given a sequence $\{L_l\}_{l\in\N}$ of natural numbers with
  $L_l\to\infty$, there exists a sequence $\{Q_l\}_{l\in\N}\subset\cK(\cH)$
  of finite rank orthogonal projections which is a proper F\o lner sequence for
  $\cT$ and $\dim Q_l\cH\geq L_l$, $l\in\N$.
\item[(iii)]
$\{P_n\}_{n\in \N}$ is a F{\o}lner sequence for $\cT$ if and only if
the following condition holds:
\begin{equation}\label{F1}
\lim_{n} \frac{\|T P_n-P_n T\|_1}{\|P_n\|_1} = 0\; .
\end{equation}
\end{itemize}
\end{proposition}
\begin{proof}
(i) It is obvious that $\{P_n\}_{n\in \N}$ is a F\o lner sequence for $\cT$ if and only if it is a F\o lner sequence for
$C^*(\cT,\1)$. Moreover, if $\{P_n\}_{n\in \N}$ is proper, then it also
satisfies $\|KP_n-K\|\to 0$ for any $K\in\cK(\cH)$. This implies that
$\|KP_n-P_n K\|\to 0$, i.e., $\{P_n\}_{n\in \N}$ quasidiagonalizes any compact
operator and, therefore, $\{P_n\}_{n\in \N}$ is also a proper F\o lner sequence for $T+K$
for any $T\in \cT$, $K\in \cK(\cH)$.

For (ii), just notice that
we can choose an increasing subsequence $Q_l=P_{n_l}$ such that
$\dim Q_l\cH\geq L_l$ and $\lim_{l\to\infty} Q_l=\1$. Then
Eq.~(\ref{eq:F1}) will be satisfied replacing $P_n$ by $Q_n$.

(iii) By item (i) we have that $\{P_n\}_{n\in\N}$ is a F{\o}lner sequence for
$\cT$ if and only if it is a F{\o}lner sequence for $C^*(\cT,\1)$ and
we can apply Lemma~1 in \cite{Bedos97}.
\end{proof}

%
If $P$ and $Q$ are orthogonal projections, we will denote by
$P\vee Q$ the orthogonal projections onto the closure of
$P\cH+Q\cH$. Finally, we will use the common notation for the
commutator of two operators: $[A,B]:=AB-BA$.

Next we give two useful formulations of the existence of a proper F\o lner sequence.

\begin{proposition}\label{local-formulation}
Let $T\in\cL(\cH)$ with $\dim\cH=\infty$. Then
the following assertions are equivalent.

\begin{enumerate}
\item[(i)] $T$ has a proper F\o lner sequence.

\item[(ii)]
For each $\eps >0$ and each $n \in \N$ there exists a projection $P\in \cPfin(\cH)$
such that $\rank P\ge n $ and \enspace $\phi(T,P)< \eps$
(see \eqref{def-phi}).

\item[(iii)]
For each $Q \in \cPfin(\cH)$ and each $\eps >0$ there exists a projection $R\in \cPfin(\cH)$
such that $R\ge Q $ and $\phi(T,R)< \eps$.

\end{enumerate}

\end{proposition}

\begin{proof}

(i) $\implies$ (ii). This is by the definition of a proper F\o lner sequence.

(ii) $\implies$ (iii). Suppose (ii) holds. Let $Q\in \cPfin(\cH)$ and $\eps>0$ be given.
By (ii), there exists a projection $P\in \cPfin(\cH)$ such that $\phi(T,P)<\eps/2$ and
\[
\|T\| \cdot \|Q\|_2  \le \frac \eps 4 \,\|P\|_2.
\]
Put $R:=P\vee Q$ and note that $\|P\|_2\le \|R\|_2$, $R\ge Q$ and $R\ge P$.
We assert that $R$ has the desired properties.
First notice that
$\rank R - \rank P\le \rank Q$, which implies that $\|R-P\|_2\le \|Q\|_2$.
So we get
\begin{multline*}
\|[T,R]\|_2 \le
\|[T,P]\|_2 + \|[T,R-P]\|_2
\le
\|[T,P]\|_2 + 2 \|T\| \cdot \|R-P]\|_2 \\
\le
\phi(T,P)\|P\|_2 + 2 \|T\| \|Q\|_2
<
\frac \eps 2 \, \|R\|_2+\frac \eps 2 \, \|P\|_2 \le \eps  \|R\|_2 ,
\end{multline*}
which yields $\phi(T,R)< \eps$.

(iii) $\implies$ (i). Suppose (iii) holds. Choose any sequence $\{L_n\}$ of
non-zero finite dimensional orthogonal projections such that $L_1\le L_2\le \dots \le L_n \le \dots$ and
$s-\lim L_n=\1$. Then a proper F\o lner sequence $\{P_n\}$ for $T$ can be constructed inductively as follows.
Take $P_1=L_1\in \cPfin(\cH)$. If $P_1, \dots, P_n$ have been defined, use (iii) to choose
$P_{n+1}\in \cPfin(\cH)$ that satisfies
$P_{n+1}\ge P_n\vee L_n$ and $\phi(T, P_{n+1})\le \frac 1 {n+1}$.
Then $P_{n+1}\ge P_n$,
$P_{n+1}\ge L_n$ for any $n$ and $\phi(T, P_n)\to 0$ as $n\to \infty$, which
implies that $\{P_n\}$ is a proper F\o lner sequence.
\end{proof}

\begin{remark}\label{rem:projections}

\begin{itemize}

\item[(i)]
The preceding proposition also shows that given an operator $T$ and
a sequence of finite rank projections
$\{Q_n\}_n$ such that $\dim Q_n$ is unbounded and $\phi(T,Q_n)\to 0$ one can
construct a proper F\o lner sequence for $T$ in the sense of Definition~\ref{def:Foelner}.

\item[(ii)]
The equivalent formulations stated above are usual when dealing with asymptotic properties. Adapting from the context of
quasidiagonality (cf.~\cite[Section~7.2]{bBrown08}) one can also prove that $T$ has a proper F\o lner sequence if and only if
for each $\eps >0$ and each finite set $\cF\subset\cH$ there exists a $P\in \cPfin(\cH)$ such that
$\phi(T,P)< \eps$ and $\|Px-x\|<\eps$ for all $x\in\cF$.
\end{itemize}

\end{remark}

The preceding result immediately implies that the existence of a
proper F\o lner sequence for a direct sum can be localized in one of the
direct summands.

\begin{proposition}\label{localization}
Let $\cH$ and $\cH'$ be separable Hilbert spaces with $\dim\cH=\infty$.
If $T$ has a proper F\o lner sequence, then $T\oplus X\in\cL(\cH\oplus\cH')$ has a proper F\o lner sequence for {\em any}
$X\in\cL(\cH')$.
\end{proposition}

\begin{proof}
Assume that  $T$ has a proper F\o lner sequence and $X$ is any operator in on $\cH'$. Then for each $\eps >0$ and each $n \in \N$
there exists a $P\in \cPfin(\cH)$
such that $\rank P\ge n $ and \enspace $\phi(T,P)< \eps$. Then
$\phi(T\oplus X,P\oplus 0)=\phi(T,P)< \eps$, which shows that
$P\oplus 0$ satisfies
the properties required in Proposition~\ref{local-formulation}~(iii)
(with respect to $T\oplus X$, instead of $T$).
\end{proof}

The following proposition concerning tensor products of operators
follows from Proposition~2.13 in
\cite{Bedos95}. For convenience of the reader we give an elementary proof of it.

\begin{proposition}
\label{tens-prods}
If $\cH$, $\cK$ are Hilbert spaces and $A\in \cL(\cH)$ and $B\in \cL(\cK)$  are linear operators that have proper F\o lner sequences, then
the tensor product $A\otimes B\in \cL(\cH\otimes \cK)$ also has a  proper F\o lner sequence.
More precisely, if $\{P_n\}_n$ is a proper F\o lner sequence for $A$ and $\{Q_n\}_n$ is a proper F\o lner sequence for $B$, then
$\{P_n \otimes Q_n\}_n$ is a proper F\o lner sequence for $A\otimes B$.
\end{proposition}

\begin{proof}
Suppose that $\{P_n\}_n$ and $\{Q_n\}_n$ are as above. Then $\{P_n \otimes Q_n\}_n$ is an increasing sequence
strongly converging to the identity. Moreover we have
\begin{align*}
(A\otimes B)(P_n\otimes Q_n)
 -
(P_n\otimes Q_n)(A\otimes B) & =
(AP_n)\otimes (BQ_n) -
(P_nA)\otimes (Q_nB) \\
& =
[A, P_n] \otimes (BQ_n) +
(P_nA)\otimes [B, Q_n].
\end{align*}
This equality and the property $\|C\otimes D\|_2=\|C\|_2\|D\|_2$ imply that
\begin{align*}
\phi\big(A\otimes B, P_n\otimes Q_n\big)
& =
\frac
{\|[A\otimes B, P_n\otimes Q_n]\|_2}
{\|P_n\|_2\, \|Q_n\|_2}
\\
& \le
\frac {\|[A, P_n]\|_2} {\|P_n\|_2 } \cdot \|B\| +
\|A\| \cdot \frac {\|[B, Q_n]\|_2} {\|Q_n\|_2} \to 0
\qquad \text{as} \enspace n\to\infty.                        \qedhere
\end{align*}
\end{proof}

The existence of a F\o lner sequence for a unital C*-algebra
has important structural consequences.
For the next result we need to recall the following notion: a state $\tau$ on
the unital C*-algebra $\cA\subset\cL(\cH)$
(i.e., a positive and normalized linear functional on $\cA$) is called an
\textit{amenable trace}
if there exists a state $\psi$ on $\cL(\cH)$ such that
$\psi\upharpoonright\cA=\tau$ and
\begin{equation*}
\psi(X A) = \psi(A X)\;,\quad X\in \cL(\cH)\;,\;A\in\cA\,.
\end{equation*}
Note that the previous equation already implies that $\tau$ is a trace
on $\cA$. The state $\psi$ is also referred in the literature as a
hypertrace for $\cA$. Hypertraces are the algebraic analogue of the invariant mean mentioned
at the beginning of the Introduction.
Later we will need the following standard result.
(See \cite{Connes76,ConnesIn76} for
the original statement and more results in the context of operator
algebras; see also \cite{Bedos95,AraLledo12} for additional results in the
context of C*-algebras related to the existence of a hypertrace.)
\begin{proposition}\label{pro:hypertrace}
  Let $\cA\subset\cL(\cH)$ be a separable unital C*-algebra. Then
 $\cA$ has a F\o lner sequence if and only if $\cA$ has an amenable trace.
\end{proposition}

In general, it is not true that if $\cA$ has an amenable trace, then it must also
have a {\em proper} F\o lner sequence.

Finally we recall the following definition from the Introduction.
\begin{definition}
$T\in\cL(\cH)$ is called a finite operator if
\[
 0\in\overline{W\left([T,X]\right)}\quad\mathrm{for~all}\quad X\in\cL(\cH)\;,
\]
where $W(T)$ denotes the numerical range of the operator $T$, i.e.,
\[
W(T)=\{\langle Tx,x\rangle\mid x\in\cH\;\;,\;\;\|x\|=1 \}\;,
\]
and where the bar means the closure of the corresponding subset in $\C$.
\end{definition}

In this context the following class of operators plays a distinguished role:

\begin{definition}\label{def:finite-block}
Let $T\in \cL(\cH)$. We say that $T$ is
finite block reducible if $T$ has a non-trivial finite-dimensional reducing
subspace, i.e., there is an orthogonal decomposition $\cH=\cH_0\oplus \cH_1$,
which reduces $T$, where $\cH_0$ is finite
dimensional and non-zero.
\end{definition}

We collect in the following theorem some standard results due to Williams about the class of finite operators
(cf.~\cite{Williams70}).

\begin{vartheorem}[Williams]
An operator $T\in\cL(\cH)$ is finite if and only if
$C^*(T,\1)$ has an amenable trace. The class of finite operators
is closed in the operator norm and contains all finite block reducible
operators.
\end{vartheorem}

It follows that the closure of the set of all
finite block reducible operators is contained in the class of finite operators.

Combining Williams' Theorem with Proposition~\ref{pro:hypertrace}, we get the following fact.

%
%
\begin{corollary}
\label{cor:fin=Fol_seq}
For any operator $T\in \cL(H)$, the following properties are equivalent:
\margp{Si lo escribe Williams, habr\'a que cambiarlo}

\begin{itemize}
\item[(i)] $T$ is finite;

\item[(ii)] $T$ has a F\o lner sequence;

\item[(iii)] $C^*(T, \1)$ has an amenable trace.
\end{itemize}
\end{corollary}

\begin{remark}\label{rem:subtle}
Note that in the reverse implication of Proposition~\ref{pro:hypertrace} the sequence of projections
does not have to be a proper F\o lner sequence in the sense of Definition~\ref{def:Foelner}.
In fact, one can easily construct the following counterexample: consider
a finite block reducible operator $T= T_0\oplus T_1$
on the Hilbert space $\cH=\cH_0\oplus\cH_1$, with $\mathrm{dim}H_0<\infty$ and
where $T_1$ has no F\o lner sequence (examples of these type of operators
are given in Section~\ref{sec:examples}). Then, by
Williams theorem it follows that $C^*(T\,,\,\1)$ has a
hypertrace and by Proposition~\ref{pro:hypertrace} it has
a F\o lner sequence also. The simplest choice of F\o lner sequence is the
constant sequence $P_n=\1_{\cH_0}\oplus 0$ which trivially
satisfies (\ref{eq:F1}) for $T$. But $T$ cannot have a proper F\o lner sequence
because $T_1$ has no F\o lner sequence
(see Proposition~\ref{prop:decomp} below).


\end{remark}

\section{Strongly non-F\o lner operators}\label{sec:NF}

In the present section we begin the analysis of operators with no F\o lner sequence.

\begin{definition}\label{def:strongNfol}
Let $\cH$ be an infinite dimensional Hilbert space
and $T$ an operator on $\cH$. We will say that $T$  is {\em strongly non-F\o lner}
if there exists an $\eps >0$ such that all projections $P\in\cPfin(\cH)$ satisfy
\[
 \phi(T, P)\ge \eps\;.
\]
\end{definition}

\begin{theorem}
\label{thm:str-foln}
Let $T\in\cL(\cH)$ with $\dim \cH=\infty$. Then $T$
has no proper F\o lner sequence if and only if $T$ has an orthogonal sum
representation $T=A\oplus \wt T$ on $\cH=\cH_0\oplus\HH$, where $\dim \cH_0 < \infty$
(so that  $A$ is a finite dimensional operator) and $\wt T$ is strongly non-F\o lner.
\end{theorem}

To prove the preceding theorem we will need the following lemmas that involve projections.

\begin{lemma}
\label{prp:WOT}
Let  $\{P_n\}_{n \in\N}$ and $\hatP$ be
orthogonal projections in $\cH$.
If $\hatP$ has finite rank and the sequence $\{P_n\}_{n \in\N}$ tends to zero in the strong
operator topology, i.e. $P_n\strongtends 0$ as $n\to \infty$,
then $\|\hatP P_n\|\to 0$.
\end{lemma}

\begin{proof}
It suffices to prove the assertion for the case when
$\hatP$ has rank one: let $\hatP = ff^*$
for a unit vector $f$ in $\cH$. Then
$\|\hatP P_n\|\leq\|f\|\|P_nf\| =\|P_nf\| \to 0$ as $n\to \infty$.
\end{proof}

\begin{lemma}
\label{lem:3s^2}
Let $s \in \N$,  $P_1, \dots, P_s\in \cPfin(\cH)$,
$j=1, \dots, s$. If $\|P_j P_k\|\le \de := 1/(3 s^3)$ for all
indices $j\ne k$, then
the ranges of projections $P_j$ are linearly independent spaces and
$$
P_1+\dots+P_s \ge \frac 12  \Big(P_1 \vee\dots \vee P_s\Big)\;.
$$
\end{lemma}

\begin{proof}
Let $f\in \Big(P_1\vee\dots\vee P_s\Big)\cH$. Then $f=\sum_j f_j$,
where $f_j\in P_j\cH$. If either $j\ne k$ or $j\ne \ell$, then
$\big|\langle P_j f_k, f_\ell \rangle\big|\le \de \|f_k\|\|f_\ell\|$.
Hence
\begin{multline*}
\sum_{j=1}^s \;
\langle
P_j f, f
\rangle
=
\sum_j
\sum_{k, \ell}
\langle
P_j f_k, f_\ell
\rangle
=\sum_k \|f_k\|^2
 +\sum_j\sum_{\scriptsize
              \begin{array}{c}
            j\ne k \; \mathrm{or}\\
            j \ne \ell
               \end{array}
}
\mathrm{Re}\;\langle
P_j f_k, f_\ell
\rangle
\\
\ge
\sum_k \|f_k\|^2
-s \de
\Big(\sum_k \|f_k\|\Big)^2
\ge
(1-s^3\de)\sum_k \|f_k\|^2
=
\frac 23
\sum_k \|f_k\|^2.
\end{multline*}
In particular, if $f=0$, then $f_j=0$ for all $j$. Hence the ranges of $P_j$ are
linearly independent.

Similar arguments show that $\sum_j\|f_j\|^2\ge \frac 1{(1+s^2\de)}\, \|f\|^2$.
Combining the last two inequalities and since $\delta= \frac{1}{3 s^3}$
we obtain the estimate
$\sum_j
\langle
P_j f, f
\rangle \geq \frac{1}{2}\|f\|^2
$
which proves the last statement.
\end{proof}

\begin{lemma}
\label{lem:close-projs}
If $P,Q\in \cPfin(\cH)$ and $L\in\cL(\cH)$, then
$$
\big|
\phi(L,P)-\phi(L,Q)
\big|
\le
4\,\|L\|\cdot\,
\frac {\|P-Q\|_2}
{\max\big(\|P\|_2,\|Q\|_2\big)}\, .
$$
\end{lemma}

\begin{proof}[Proof]
Without loss of generality, let us assume that $\|P\|_2\le \|Q\|_2$.
Then we have
\begin{align*}
\|Q\|_2\, \|P\|_2  \big| \phi(L,P)-\phi(L,Q) \big|
& =
\Big|
\|Q\|_2\, \|[L,P]\|_2 - \|P\|_2\, \|[L,Q]\|_2
\Big| \\
& \le
\big|\|Q\|_2-\|P\|_2\big|\cdot
\|[L,P]\|_2 + \|P\|_2  \, \|[L,Q-P]\|_2 \\
& \le 4\, \|L\|\,\|Q-P\|_2 \|P\|_2,
\end{align*}
which implies the desired estimate.
\end{proof}

\begin{proposition}\label{prop:decomp}
Let $T= A\oplus \wt T$ on $\cH=\cH_0\oplus\HH$, where $\dim \cH_0 < \infty$ (hence $A$ is a finite dimensional operator).
Then $T$ has a proper F\o lner sequence if and only if $\wt T$ has a proper F\o lner sequence.
\end{proposition}
\begin{proof}
The implication ``$\Leftarrow$'' follows from Proposition~\ref{localization}. To prove the implication ``$\Rightarrow$''
suppose that $T$ has a proper F\o lner sequence and
put $d:=\dim \cH_0 < \infty$. For any $\varepsilon$ and any
$N > d$ there exists a $P\in \cPfin(\cH)$ such that  $\rank P\ge N $ and $\phi(T,P)< \eps$.
Now, for each such $P$ there exists also a $\hatP\in \cPfin(\HH)$ such that $0\oplus\hatP\leq P$ and
$\rank \hatP + d \geq \rank P$. (Take as $\hatP$, e.g., the orthogonal projection onto $P\cH\cap (0\oplus\HH)$.)
Using Lemma~\ref{lem:close-projs} we get
\begin{eqnarray*}
 \phi\big(\wt T,\hatP\big) &=& \phi\big(T,0\oplus\hatP\big)\;\leq\;\phi(T,P)+\big|\phi(T,P)-\phi\big(T,0\oplus\hatP\big)\big|\\
                   &\leq& \phi(T,P) +\frac{ 4\|T\|\;\|P-(0\oplus\hatP)\|_2}{\|P\|_2}\\
                   &\leq& \phi(T,P) +\frac{ 4\|T\|\;d^\frac12}{\|P\|_2}      < 2 \eps\;,
\end{eqnarray*}
where for the last inequality we have chosen $P$ so that
$
\|P\|_2 > \frac{4\|T\|d^\frac12}{\eps}.
$
By Proposition~\ref{local-formulation}~(ii) it follows that $\wt T$ has also a proper F\o lner sequence.
\end{proof}

\begin{proposition}\label{prop:splitting}
Let $T\in\cL(\cH)$ and suppose that $\phi(T,P)\not=0$ for all $ P\in \cPfin(\cH)$. If there is a
F\o lner sequence of projections $\{P_n\}_n\subset \cPfin(\cH)$ of a constant rank,
%
then $T$ has a proper F\o lner sequence.
\end{proposition}
\begin{proof}
Let $\{P_n\}_n$ be a sequence of projections such that
the rank $m:= \rank P_n$ is constant and non-zero.
We can represent them as $P_n=f_nf_n^*$, where $f_n\colon\C^m\to \cH$
are isometries. Moreover, by weak compactness of the unit ball in $\cH$
there is a contraction $g\colon\C^m\to \cH$, $\|g\|\le 1$, such that
(passing possibly to a subsequence)
\beqn
\label{g}
f_n \weaktends g\;.
\neqn
First we prove that $g=0$. For this, suppose that $g\ne 0$
and we will show that this leads to a contradiction. If $g\ne 0$
there exists some $k$, with $1\le k \le m$, and an
isometry $g_0\colon\C^k\to \cH$ such that $\Ran g=\Ran g_0$.
Notice that
$$
(I-P_n)T P_n=
\big(
T f_n- f_n(f_n^* T f_n)
\big)f_n^*.
$$
Put $\alpha_n = f_n^* T f_n \colon\C^m\to \C^m$ and
\beqn
\label{h-n}
h_n= T f_n - f_n \alpha_n\colon \C^m\to \cH,
\neqn
so that
$$
(I-P_n)T P_n=h_nf_n^*.
$$
Since projections $P_n$ have constant rank and $\phi(T, P_n)\to 0$ it follows that
\[
\|(I-P_n)T P_n\| \to 0 \quad\mathrm{as}\quad n\to \infty\;.
\]
Thus $\|h_n\| \to 0$. Passing possibly to a subsequence,
we can assume that there is a limit
$\lim_{n\to \infty} \alpha_n=\alpha\in \cL(\C^m)$.
By \eqref{h-n}, $f_n\alpha_n=T f_n-h_n$, so by \eqref{g}, we get
$$
f_n \alpha_n \weaktends T g \qquad \text{as}\enspace n\to \infty.
$$
By applying \eqref{g} once again, it follows that $g\alpha= T g$. In particular,
$\Ran(T g)\subset \Ran g$.
Notice that $P_{g_0} := g_0g_0^*\not=0$ is the orthogonal projection onto
$\Ran g_0$. Since $\Ran g = \Ran g_0$, we arrive at
the equality
$(I-P_{g_0})T P_{g_0}=0$.
In the same way, we can prove that
$(I-P_{g_0}) T^* P_{g_0}=0$ (for the same isometry $g_0$).
Hence $\phi\big(T, P_{g_0}\big)=0$, which
contradicts the assumption. Therefore we must have
$g=0$, that is, $f_n\weaktends 0$ as $n\to \infty$.
Hence $\big| \langle P_n a, b\rangle \big|
\le \|f_n^*a\|\|f_n^*b\|\to 0$ for any $a,b\in \cH$,
that is, $P_n\weaktends 0$, hence $P_n\strongtends 0$.

To show that $T$ has a proper F\o lner sequence let $\eps>0$ and $N\in\N$.
Consider also
a positive $\de < \min\Big(\frac{\eps}{4N}\,,\, \frac{1}{3N^3} \Big)$.
From the assumption $\phi(T,P_n)\to 0$
and Lemma~\ref{prp:WOT} we
can choose projections $P_{n_1}, P_{n_2}, \dots, P_{n_N}$ from
the sequence $\{P_n\}_n$ that satisfy
$$
\phi(T, P_{n_j})<\de
\quad  \text{and} \quad \|P_{n_j} P_{n_k}\|<\de
$$
for all indices $j\ne k$, $1\le j, k\le N$.
To simplify notation we will write $P_j$ instead of $P_{n_j}$. Put $P:=(P_1\vee \dots \vee P_N)$.
Since $(I-P)(I-P_j)=(I-P)$ we have that $\|(I-P)T P_j\|\leq\|(I-P_j)T P_j\|<\de$
for all $j=1, \dots, N$.
Hence $\big\|(I-P)T \big(\sum_j P_j\big)\big\|<N\de$. Finally,
Lemma~\ref{lem:3s^2} implies that
$\sum_j P_j\ge \frac 12 P$.
%
%
We denote
the inverse of $\sum_j P_j$
on $P\cH$ by $Q$. Then $\|Q\|\le 2$
and $P=\big(\sum_j P_j\big) Q$. This gives that
$\|(I-P)T P\|<2N\de$. In the same way we show that
$\|PT (I-P)\|<2N\de$, and, therefore,
$\phi(T, P)<4N\de< \eps$ and by
Lemma~\ref{lem:3s^2}, $\rank P\ge N$. From Proposition~\ref{local-formulation}
we conclude that $T$ has a proper F\o lner sequence.
\end{proof}

With the preceding material we can now prove the main result of this section:
\begin{proof}[Proof of Theorem~\ref{thm:str-foln}]
If $T\in\cL(\cH)$ has an orthogonal sum
representation $T=A\oplus \wt T$ on $\cH=\cH_0\oplus\HH$, where $\dim \cH_0 < \infty$
and $\wt T$ is strongly non-F\o lner, then Proposition~\ref{prop:decomp} implies that $T$ has
no proper F\o lner sequence.

To prove the other implication of the theorem
suppose that $T$ has no proper F\o lner sequence. By
Proposition~\ref{local-formulation},
there exist some $\eps'>0$ and $M\in \N$ such that
$$
\forall P\in \cPfin(\cH)\;,
\qquad
\rank P>M \implies \phi(T,P) \geq\eps'.
$$
In particular, it follows that if $T$ decomposes as
$T= A\oplus \TT$, where $A\in \cL(\cHknot)$, with
$\cHknot$ finite dimensional, then $\dim\cHknot\le M$.
Consider a decomposition
$T=A \oplus \TT$, where
$A\colon\cHknot\to\cHknot$, $\TT\colon\HH\to \HH$, and
$\ell:= \dim \cHknot$ is the largest possible.
(The case where $\ell=0$ is not excluded.)
We prove next that $\TT$ is a strongly non-F\o lner operator:
by Proposition~\ref{prop:decomp} and since $\ell\le M$, we
have that $\TT$ has no proper F\o lner sequence.
Therefore there exist $\eps_1>0$ and $s\in \N$ such that
\beqn
\label{eps1}
\forall \hatP\in \cPfin(\HH)\;,
\qquad
\rank \hatP > s \implies \phi(\TT,\hatP) \geq\eps_1\;.
\neqn
On the other hand, since $\ell$ is the largest possible,
it follows that
\begin{equation}
\label{3stars}
\hatP\in \cPfin(\HH)\;\mathrm{with}\; \hatP\not= 0\;
\implies \phi(\TT,\hatP) \ne 0\;.
\end{equation}
We claim that \eqref{3stars} implies that
\beqn
\label{star}
\exists\; \eps_2 > 0\quad \forall\,
\hatP\in \cPfin(\HH)
\qquad
0<\rank \hatP \leq s
 \implies \phi(\TT,\hatP) \ge \eps_2.
\neqn
Then, putting $\eps=\min\{\eps_1, \eps_2\}$,  we will conclude that $\TT$ is
strongly non-F\o lner (cf.~Definition~\ref{def:strongNfol}).
So it remains to deduce assertion \eqref{star}.

Assume that \eqref{star} does not hold. Then there is some
$m$, with $1\le m\le s$, and a sequence of projections
$\{P_n\}_n\subset\cPfin(\HH)$
of rank $m$ that is a F\o lner sequence, i.e., we have
\begin{equation}\label{eq:to0}
 \lim_{n\to\infty}\phi(\TT, P_n)= 0\;.
\end{equation}
From (\ref{3stars}) and Proposition~\ref{prop:splitting}
applied to the operator $\TT$ we conclude that $\TT$
has a proper F\o lner sequence. But this contradicts
(\ref{eps1}) and, therefore, (\ref{star}) must hold.
\end{proof}

\section{Relation between proper F\o lner sequences and finite operators}\label{sec:relation}

We show in this section a useful characterization of finite operators
that involves proper F\o lner sequences. Recall the definitions
and results stated at the end of Section~\ref{sec:foelner}.

\begin{theorem}\label{teo:williams-foelner}
Let $T\in\cL(\cH)$. Then, $T$ is a finite operator if and only if $T$ is finite block
reducible or $T$ has a proper F\o lner sequence.
\end{theorem}
\begin{proof}
(i) If $T$ is finite block reducible, the $T$ is a finite operator (cf.~\cite{Williams70}).
Moreover, if $T$ has a proper F\o lner sequence, then the C*-algebra $C^*(T,\1)$ has the same proper F\o lner sequence
and, by Proposition~\ref{pro:hypertrace}, it also has an amenable trace. Then,
by Williams' theorem (see also Theorem~4 in \cite{Williams70}) we conclude that $T$
is finite.

(ii) To prove the other implication, assume
$T$ is a finite operator. We consider several cases. If there exists a (non-zero)
$P\in\cPfin(\cH)$ such that $\phi(T,P)=0$, then since $\phi(T,P)=\frac{\|[T,P]\|_2}{\|P\|_2}$
we must have
%
%
$[T,P]=0$. This shows that $T$ is finite block reducible.
Consider next the situation where $\phi(T,P)\not=0$ for all $P\in\cPfin(\cH)$. Since $T$
is finite we can use again Theorem~4 in \cite{Williams70} to conclude that $C^*(T,\1)$
has an amenable trace. Applying Proposition~\ref{pro:hypertrace} 
(see also Theorem~1.1 in \cite{Bedos95}) we conclude that there exists a F\o lner sequence
of non-zero finite rank projections $\{P_n\}_n$, i.e., we have
\[
 \lim_{n\to\infty} \phi(T,P_n) = 0\;.
\]
(Note that $P_n$ is not necessarily a proper
F\o lner sequence in the sense of Definition~\ref{def:Foelner}; cf.~Remark~\ref{rem:subtle}.)
Two cases may appear:
if $\mathrm{dim}\,P_n\cH\leq m$ for some $m\in\N$, then choose a subsequence with constant rank
and by Proposition~\ref{prop:splitting} we conclude that $T$ has a proper F\o lner sequence.
If $\mathrm{dim}\,P_n\cH$ is not bounded, then from Remark~\ref{rem:projections}~(i) we also
have that $T$ has a proper F\o lner sequence.
\end{proof}

\

From Theorem~\ref{teo:williams-foelner} and taking into account the classification
of operators described in Table~1 of the Introduction we have the following result.

\begin{corollary}\label{cor:yes-no}
Let $T\in\cL(\cH)$. Then
\begin{itemize}
 \item[(i)] $T$ is a finite operator if and only if $T$ is in one of the following mutually
disjoint classes: $\mathcal{W}_{0+}$, $\mathcal{W}_{0-}$, $\mathcal{W}_{1+}$.
 \item[(ii)] $T$ is not a finite operator (i.e., it is of class $\mathcal{S}$) if and only if $T$ is strongly non-F\o lner.
 \item[(iii)] The class of strongly non-F\o lner operators is open and dense in $\cL(\cH)$.
\end{itemize}
\end{corollary}
\begin{proof}
The characterization of finite operators and its complement stated in (i) and (ii)
follows from Theorem~\ref{teo:williams-foelner}
and Williams' theorem at the end of Section~\ref{sec:foelner}.
To prove part (iii) we use that the class of finite
operators is closed and nowhere dense (cf.~\cite{Herrero89}).
Therefore the set of strongly non-F\o lner operators is an
open and dense subset of $\cL(\cH)$.
\end{proof}

\begin{remark}
\margp{NUEVO}
In fact, the assertion that the class of strongly non-F\o lner operators is open in the norm
topology follows easily from our definition of this class. Indeed, let $T$ be strongly
non-F\o lner operator, so that there is an $\eps>0$ such that
$\phi(T,P)\ge \eps$ for all $P\in \cPfin(\cH)$. It is easy
to see that $\big|\phi(T,P)-\phi(T',P)\big|\le 2 \|T-T'\|$ for any operator $T'$. Hence
any operator $T'$ with $\|T-T'\|< \eps/2$ is strongly non-F\o lner.
So an application of Corollary~\ref{cor:yes-no}
gives an alternative proof of the result by Williams \cite{Williams70}
that the set of finite operators is closed.
\end{remark}

\section{Classes of non-normal operators with a proper F\o lner sequence}\label{sec:with-FS}
%
%

In the present section we single out several classes of operators
such that any operator in these classes has a proper F\o lner sequence.
The unilateral shift is a basic example that shows the difference between
the notions of proper F\o lner sequences and quasidiagonality.
It is a well-known fact that the unilateral shift $S$
is not a quasidiagonal operator. (This was shown by Halmos in \cite{Halmos68}; in
fact, in this reference it is shown that $S$ is not even quasitriangular.)
In the setting of abstract C*-algebras it can also be shown that a C*-algebra
containing a proper (i.e.~non-unitary) isometry is not quasidiagonal
(see, e.g., \cite{BrownIn04,bBrown08}).

On the other hand, $S$ has a canonical proper F\o lner sequence. Indeed,
let $S$ be defined on $\cH:=\ell^2(\N_0)$ by $Se_i:=e_{i+1}$, where
$\{e_i\mid i=0,1,2,\dots\}$ is the canonical basis of
$\ell^2(\N_0)$. Then it is very easy to see that orthogonal projections
$P_n$ onto span$\{e_i\mid i=0,1,2,\dots, n\}$ form a proper F\o lner sequence for $S$.
We will see later in this section that, in fact, any isometry
has a proper F\o lner sequence.

We recall some standard definitions.
An operator $T\in \cL(\cH)$ is called
\textit{hyponormal} if its self-commutator $[A^*,A]$ is
nonnegative. $T$ is called \textit{essentially hyponormal}
if the image in the Calkin algebra $\cL(\cH)/\cK(\cH)$ of
$[T^*,T]$ is a nonnegative element, that is, if $[T^*,T]$ is a sum
of a nonnegative and a compact selfadjoint operator. Next,
$T$ is said to be \textit{essentially normal} if
$[T^*,T]$ is compact (that is, $[T^*,T]$ is zero as an element of the Calkin
algebra). Finally, $T$ is called \textit{quasinormal}
if $T$ and $T^*T$ commute.
Any isometry is quasinormal, any quasinormal operator is
subnormal and any subnormal operator is hyponormal,
see \cite[Chapter II]{bConway91}.

\begin{theorem}\label{main1}
Any essentially hyponormal operator $T\in\cL(\cH)$ has a proper F\o lner sequence.
\end{theorem}

\begin{proof}
Let $T$ be essentially hyponormal.
By \cite[Theorem~5]{Williams70} and the discussion that follows this theorem,
$T$ is a finite operator. Consider all finite-dimensional reducing subspaces $\cH_0$ of $T$ (including
the zero one).

There are two possibilities.

1) Suppose that among these subspaces there is one of largest dimension, say, $\cH_0$, and $T=T_0\oplus T_1$ with
respect to the corresponding decomposition $\cH=\cH_0\oplus\cH_0^\perp$.
Then $T_1$ is essentially hyponormal and not finite block reducible.
By Theorem~\ref{teo:williams-foelner}, $T_1$ has a proper F\o lner sequence.
Therefore $T$ also has a proper F\o lner sequence.

2) Now suppose that, to the opposite, the dimensions of these
subspaces can be arbitrarily large.
Then we deduce from Proposition~\ref{local-formulation} that
in this case, too, $T$ has a proper F\o lner sequence.
\end{proof}


\begin{corollary}\label{cor:3classes}
Every essentially normal operator
(that is, an operator $T$ such that $[T^*,T]\in \cK(\cH)$)
has a proper F\o lner sequence.
Every hyponormal operator (in particular, any subnormal, any quasinormal and any isometry)
also has a proper F\o lner sequence.
\end{corollary}

\begin{remark}
For some of the above operator classes, one can give alternative direct proofs.
\begin{enumerate}
\item[(i)]
Any isometry $V\in\cL(\cH)$ has a proper F\o lner sequence. Indeed,
without loss of generality we may assume that $V$ is not unitary.
By Wold's decomposition theorem (cf.~\cite[Section~V.2]{bDavidson96}) we have that
$
 V\cong S\oplus A\;,
$
where $\cong$ means unitary equivalence, $S$ is the unilateral shift and
$A=\big(\mathop{\oplus}\limits_{i=0}^n S\big)\oplus U$ for some cardinal number $n$ and unitary $U$.
Since we showed above that $S$ has a canonical
proper F\o lner sequence we can apply Proposition~\ref{localization} and the proof is concluded.

\item[(ii)]
Any quasinormal operator has a proper F\o lner sequence.
To see that, notice first that by Brown's theorem (see \cite[Theorem~II.3.2]{bConway91}), any
quasinormal operator is unitarily equivalent to $N\oplus (A\otimes S)$, where
$N$ is normal, $A$ is non-negative and $S$ is the unilateral shift.
Next, as we mentioned already, $S$ has an explicit proper F\o lner sequence and since $A$
is self-adjoint it has a proper F\o lner sequence too.
By Proposition~\ref{tens-prods}, $A\otimes S$ has a proper F\o lner sequence.
Finally, Proposition~\ref{localization} implies that $N\oplus (A\otimes S)$ has
a proper F\o lner sequence.

\item[(iii)]
One can also give an alternative proof of the fact that
any essentially normal operator has a proper F\o lner sequence by applying the
Brown--Douglas--Fillmore theory (see \cite{BrownIn73,Davidson-ess}) and a model
for essentially normal operators, similar to that given on p.~122 of the cited work.
\end{enumerate}
\end{remark}




The next result refers to the existence of a proper F\o lner sequence for the class
of Toeplitz operators on the $d$-dimensional torus. Denote the
unit torus by $\T=\{z\in \C\mid |z|=1\}$.
We also recall that, given a function
$F\in L^\infty(\T^d)$, the Toeplitz operator $T_F$ on
the classical Hardy space $H^2(\T^d)$ is defined by
$T_F g=P_+(F\cdot g)$, $g\in H^2(\T^d)$,
where $P_+$ stands for the orthogonal projection from $L^2(\T^d)$ onto $H^2(\T^d)$.
Note that even for $d=1$, there are Toeplitz operators which are not essentially normal
(for instance, $T_\theta$ for any non-rational inner function $\theta$).
Using the same idea, it is easy to give an example of a Toeplitz operator $T_F$ on
$H^2$ such that neither $T_F$ nor $T_F^*$ is essentially hyponormal.

It is an easy consequence of Proposition~\ref{tens-prods} and Weierstra\ss' theorem that
any Toeplitz operator with continuous symbol on $\T^d$ has a proper F\o lner sequence.
The following more general result is also true. Its proof is similar
to the one given in Example~7.17 of \cite{bHagen01}, where only the case $d=1$ was considered.
For convenience of the reader give a brief sketch of the proof.

\begin{proposition}
Any Toeplitz operator $T_F$ on $H^2(\T^d)$ with any symbol $F$ in $L^\infty(\T^d)$
has a proper F\o lner sequence.
\end{proposition}

\begin{proof}
It is enough to consider the proof for $d=2$ and other cases follow similarly.
For any $F\in L^\infty(\T^2)$ we consider its decomposition
\begin{equation}\label{Fnorm}
F=\sum_{k_1,k_2\in\Z} a_{k_1k_2} \; e_{k_1 k_2}
 \quad\mathrm{and}\quad
\|F\|^2= \sum_{k_1,k_2\in\Z} |a_{k_1k_2}|^2 < \infty\;.
\end{equation}
where $\{e_{k_1 k_2}\}$ is the
canonical basis of $L^2(\T^2)$:
$e_{k_1 k_2}(z_1,z_2)=z_1^{k_1}z_2^{k_2}$.
Denote by $P_N$ the orthogonal projection onto $\mathrm{span}\,\{e_{k_1 k_2}\mid k_1,k_2=0\,,\dots\,, N-1\}$
and note that $\|P_N\|_2^2=N^2$. If we choose $b_N$ to be the smallest integer greater or equal than
$\sqrt{N}$, we have
\begin{eqnarray*}
\lefteqn{\frac{1}{N^2}\;
         \|(\1-P_N)\,T_F\, P_N\|_2^2 }\\[2mm]
&\le & \frac{1}{N^2}\,
\sum\limits_{\mbox{\tiny $\begin{array}{c}
                             l_1,l_2=0,\dots, N-1  \\
                             k_1\geq N, k_2\geq 0\\
                          \end{array}$}} \!\!\!\left| a_{k_1-l_1\,,\,k_2-l_2}\right|^2
\;\;+
\frac{1}{N^2}\,
\sum\limits_{\mbox{\tiny $\begin{array}{c}
                             l_1,l_2=0,\dots, N-1  \\
                             k_1\geq 0, k_2\geq N
                          \end{array}$}} \!\!\!\left| a_{k_1-l_1\,,\,k_2-l_2}\right|^2
=:A_1+A_2.
\end{eqnarray*}
Next, putting $s_j=k_j-l_j$, we get
\begin{eqnarray*}
A_1&\leq&
\frac{1}{N}\bigg(
\sum\limits_{s_1\geq 1\,,\, s_2\in\Z}
             \!\!\!\left| a_{s_1s_2}\right|^2
+\dots+
\sum\limits_{s_1\geq N\,,\, s_2\in\Z}
             \!\!\!\left| a_{s_1s_2}\right|^2
\bigg) \\[4mm]
&\leq&
\frac{1}{N}\bigg(
(N-b_N)\sum\limits_{s_1> b_N\,,\, s_2\in\Z}
             \!\!\!\left| a_{s_1s_2}\right|^2
+
b_N \sum\limits_{s_1\geq 1\,,\, s_2\in\Z}
             \!\!\!\left| a_{s_1s_2}\right|^2
\bigg) \\[4mm]
&\leq&
\sum\limits_{s_1> b_N\,,\, s_2\in\Z}
             \!\!\!\left| a_{s_1s_2}\right|^2
+\frac{b_N}{N}\,\|F\|^2\quad \mathop{\longrightarrow}\limits_{N\to\infty}\quad 0\;
\end{eqnarray*}
(see \eqref{Fnorm}). Similarly,
$A_2\to 0$ as $N\to\infty$. We get that
$\frac{1}{N^2}\, \|(\1-P_N)\,T_F\, P_N\|_2^2\to 0$. In the same way one proves that
$\frac{1}{N^2}\, \|P_N\,T_F\, (\1-P_N)\|_2^2\to 0$ as $N\to\infty$.
\end{proof}

In particular, it follows that any Toeplitz operator on $H^2(\T^d)$ is finite.

\medskip

For completeness we mention that all band-limited operators and uniform limits
of them have a proper F\o lner sequence. Recall that
a linear operator $A$ on $\cH$ is \textit{band-limited}
with respect to an orthonormal basis
$\{e_n\}_{n=1}^\infty$ or $\{e_n\}_{n=-\infty}^\infty$ in $\cH$
if there is $N\in \N$ such that the matrix elements of $A$ satisfy
$\langle Ae_j, e_k \rangle = 0$ for $|j-k|>N$
(see, e.g., \cite{Arveson94,Bedos97,pRabinovich10}). This class of operators
includes all (bounded) unilateral and bilateral weighted shifts.

Notice that
not every weighted shift is essentially normal
and not every weighted shift is quasi-diagonal. A complete description
of quasi-diagonal weighted shifts was given in \cite{Smucker82}
(see also \cite{Narayan92} for a generalization).

It is easy to see that band-limited operators can be generalized to
what we can call \textit{``acute wedge'' operators}.
%
%
%
%
By definition, $A$ has this property with respect to an orthonormal basis $\{e_n\}_{n=1}^\infty$
(or $\{e_n\}_{n=-\infty}^\infty$) in $\cH$ if there exists a
function $g(n)$ such that
$\lim_{|n|\to\infty} \frac {g(n)} {|n|^{1/2}} =0$ and $\langle Ae_j, e_k \rangle = 0$ for all $j, k$ such that $|j-k|>g(j)$.



\section{Examples of strongly non-F\o lner operators}\label{sec:examples}

\

Returning to Table~1 of the Introduction,
it easy to give examples of operators of class
$\cW_{0+}$. Next, it is immediate to see that the unilateral shift is an example of an operator in the class $\cW_{1+}$.
In this Section, we will recall several examples of operators of class $\cS$  (that is, strongly non-F\o lner operators)
and will give a new example. We remark that for any strongly non-F\o lner operator
$T$ and any operator $T_0$ acting on a non-zero finite dimensional Hilbert space, the orthogonal sum
$T_0\oplus T$ is an example of an operator in $\cW_{0-}$.

%
Halmos constructed in \cite[Theorem~5]{Halmos54} two operators $A,B\in\cL(\cH)$,
$\cH$ infinite dimensional such that $W([A,B])$ is a vertical line segment in the open right half plane.
It follows that $A$ and $B$ cannot be finite, hence both are examples of operators
which are strongly non-F\o lner.
In fact, this result was a motivational example for Williams' article \cite{Williams70}.

It is also worth mentioning that Corollary~4 in
\cite{Bunce76} gives an example of a
strongly non-F\o lner operator generating a type~$II_1$ factor.

Now let us give one more example. We will use the amenable trace that appears in
Proposition~\ref{pro:hypertrace} as an obstruction.
Recall the definition of the
Cuntz algebra $\cO_n$ (cf.~\cite{Cuntz77,bDavidson96}).
It is a C*-algebra generated by $n\geq 2$
non-unitary isometries $S_1,\ldots,S_n$ on an infinite
dimensional Hilbert space, with the property that their final
projections add up to the identity, i.e.
\begin{equation}\label{RangeProj1}
 \sum_{k=1}^{n} S_k S_k^*=\1\,.
\end{equation}
This condition implies in particular that the range projections
are pairwise orthogonal, i.e.
\begin{equation}
\label{Cuntz}
 S_l^* S_k=\delta_{lk}\1\,.
\end{equation}
The Cuntz algebra $\cO_n$ is the universal C*-algebra generated
by isometries satisfying these relations. It is easy to
realize the Cuntz algebra on the Hilbert space $\ell_2$.

\begin{proposition}\label{no-Foelner}
The Cuntz algebra $\cO_n$, $n\ge 2$, is singly generated and its generator
is strongly non-F\o lner.
\end{proposition}
\begin{proof}
By Corollary~4 (or Theorem~9) in \cite{Olsen76} any Cuntz algebra
$\cO_n$, $n\geq 2$, has a single generator $C_n$, i.e.,
$\cO_n=C^*(C_n)$. We assert that $C_n$ is strongly non-F\o lner.
\margp{He reducido la prueba; comprobarlo}
Indeed, assume that, to the contrary, it is not; then
by Corollary~\ref{cor:yes-no} (ii), $C_n$ is finite. By
Corollary~\ref{cor:fin=Fol_seq},
it would follow that $\cO_n=C^*(C_n)$ has an amenable trace $\tau$.
But this gives a contradiction since applying $\tau$ to the equations
\eqref{RangeProj1} and \eqref{Cuntz}
we obtain $n=1$.
%
%
\end{proof}

%
\begin{corollary}
There exist invertible and contractible strongly non-F\o lner operators.
\end{corollary}
\begin{proof}
From the previous examples of strongly non-F\o lner operators we can easily construct
invertible or contractible strongly non-F\o lner operators. Just note that
for any two complex constants $\lambda \ne 0$
and $\mu$, an operator $T$ is
strongly non-F\o lner if and only if the operator $\lambda T + \mu \1$ is.
\end{proof}

It is also easy to see that for any operator $T$ that has no F\o lner sequence,
there is no orthogonal basis $\{e_n\}_{n=1}^\infty$
(nor $\{e_n\}_{n=-\infty}^\infty$) in $\cH$ such that $T$ is an ``acute wedge'' operator
%
%
with respect to this basis (see the end of \S \ref{sec:with-FS}
for a definition). In \cite{pLledoYakubovich12}, we will discuss
these operators in more detail.

\medskip

The above results allow one to present the classification of Table~1 in the
following, more detailed way. Let $T\in \cL(\cH)$, and
\margp{nuevo parche}
put
\[
\ell(T):=\sup_{R\subset \cH}\; \dim R,
\]
where the supremum is taken over
all finite dimensional reducing subspaces $R$ of $T$ (including the zero one).

In the case when $\ell(T)$ is finite,
there exists a unique reducing subspace
$\cH_0$ of $T$ of dimension
$\ell(T)$ (because $R_1+R_2$ is a reducing subspace of $T$
whenever $R_1$ and $R_2$ are). In this case, $T$ decomposes as
$T=T_0\oplus T_1$ with respect to the
orthogonal sum representation $\cH=\cH_0\oplus \cH_0^\perp$.
If $0\le \ell(T)<\infty$, put
\[
\eps(T):= \inf_{P\in \cPfin( \cH_0^\perp)} \phi(T_1,P).
\]
Then we have the following cases.

\vskip.3cm
\textit{Case 1: $0\le \ell(T)<\infty$.}
For these operators, Table~1
now looks as follows:
\begin{table}[h]
\centering 
  \begin{tabular}{| c || c | c | }
    \hline
    \phantom{phantom}   &  Operators with a proper   &    Operators with no proper \\[0.5ex]
    \phantom{phantom}   &  F\o lner sequence   ($\eps(T)=0$)      &    F\o lner sequence ($\eps(T)>0$)  \\
                           \hline\hline
Finite block reducible             &                            &                          \\ 
($0 < \ell(T)<\infty$)     &   \raisebox{6pt}{$\mathcal{W}_{0+}$}       &     \raisebox{6pt}{$\mathcal{W}_{0-}$}    \\ 
\hline
Non finite block reducible &                         &       \\
($\ell(T)=0$)   &  \raisebox{6pt}{$\mathcal{W}_{1+}$}    & \raisebox{6pt}{$\mathcal{S}$}     
\\ 
    \hline
  \end{tabular}
\vskip.3cm
 \caption{} 
\end{table}


\textit{Case 2: $\ell(T)=\infty$.}
In this case, $\eps(T)$ is not defined. These are exactly block diagonal operators, which
are operators that can be decomposed into an infinite orthogonal sum of finite dimensional
operators (see, e.g., \cite[Chapter 16]{bBrown08} or \cite{Voiculescu93}).
All these operators belong to $\mathcal{W}_{0+}$.

\section{Appendix: Single generators for C*-algebras}

In the questions we are discussing, singly generated
C*-algebras of operators have some relevance. See, e.g.,
the proof of Proposition~\ref{no-Foelner}, where it was crucial
that the Cuntz algebra is singly generated. Moreover,
from Theorem~\ref{main1} and Proposition~\ref{pro-1}~(i), any
$C^*$-algebra generated by an essentially normal operator has a proper F\o lner sequence.
However, we do not know in general whether any separable C*-subalgebra $\cA$ of $\cL(\cH)$
such that $[T,S]\in\cK(\cH)$ for all $T,S\in\cA$
(or, equivalently, all operators in $\cA$ are essentially normal)
has a proper F\o lner sequence.
C*-algebras singly generated by an essentially normal operator were considered,
e.g., in \cite{FeldmanMcGuire06,FeldmanMcGuire08}.
For a nice introduction to the single generator problem we refer to \cite[Section~1]{ThielWint}
(see also \cite{Nagisa}, \cite{ChoJorg2009}).

Here we will prove the following variation of a result in \cite{Olsen76}
which might be useful when addressing questions about F\o lner sequences
for operators:
roughly, every separable C*-algebra $\cA$ of operators can be embedded into a
singly generated C*-algebra acting on a larger Hilbert space.
Notice that the word ``separable'' here
refers to separability of the metric, associated to the operator norm.
It is clear that a C*-subalgebra of $\cL(\cH)$ is separable if and only if it is countably
generated. In general if $\cA$ is unital its image under the embedding need
not be a unital C*-subalgebra of the larger algebra.

\begin{proposition}
Let $\cA$ be a unital C*-subalgebra of $\cL(\cH)$.
Then the following two assertions are equivalent.

\begin{enumerate}
\item[(i)] $\cA$ is separable.

\item[(ii)] There exists a separable Hilbert space
$\cKK=\cH\oplus \cH'$, and a singly
generated C*-subalgebra $\cB\subset \cL(\cKK)$ such that
$
\cA\oplus 0\subset \cB,
$
where $\cA\oplus 0=\big\{A\oplus 0\in \cL(\cKK)\mid A\in \cA\big\}$.
\end{enumerate}

\end{proposition}

\begin{proof}
The implication (ii) $\implies$ (i)
is obvious. To show the reverse implication denote by
$\cK(\gothh)$ the $C^*$-algebra of compact operators
on a separable Hilbert space $\gothh$.
Put $\cB=\cA\otimes \cK(\gothh)\subset \cL(\cH\otimes\gothh)$ (since $\cK(\gothh)$ is nuclear,
there is no ambiguity in the definition of the tensor product of these operator algebras;
see, e.g., \cite[Vol.~2, Chapter~11]{Kadison:Ringrose:12}).
Let $p\in \cK(\gothh)$ be a rank one
orthogonal projection onto the subspace $\langle e\rangle$ generated by a unit vector $e\in\gothh$.
The map $\Phi(a) :=  a\otimes p$ is an isometric
$*$-isomorphism from $\cA$ to $\cA\otimes\cK(\gothh)$ and define
$\cR:=\cH\otimes \gothh$.
It is clear that $\Phi(\cA)$ has the form
$\cA\oplus 0$ in the decomposition
$\cR =
\big(\cH\otimes \langle e \rangle\big)
\oplus
 \big(\cH\otimes \langle e \rangle^\perp\big)$.
Finally, by Theorem~8 in \cite{Olsen76}, $\cB$ is singly generated.
\end{proof}


\paragraph{\bf Acknowledgements:}
The first-named author wants to thank Chris Phillips for useful conversations
during his visit to CRM in January 2011. We are also indebted to Chris Phillips for
pointing to us Corollary~4 (or Theorem~9) in \cite{Olsen76}, which immediately
led to Proposition~\ref{no-Foelner}.


\end{document}